\documentclass[a4paper,10pt,leqno]{amsart}
\title{Estimates for spectral density functions of  matrices over $\IC[\IZ^d]$}

\author{L\"uck, W.}
        \address{Mathematicians Institut der Universit\"at Bonn\\
                Endenicher Allee 60\\
                53115 Bonn, Germany}
         \email{wolfgang.lueck@him.uni-bonn.de}
          \urladdr{http://www.him.uni-bonn.de/lueck}
         \date{October, 2014}
              \keywords{spectral density function, Novikov-Shubin invariants}
     \subjclass[2010]{46L99, 58J50}

\usepackage{hyperref}
\usepackage{color}
\usepackage{pdfsync}
\usepackage{calc}
\usepackage{enumerate,amssymb}
\usepackage[arrow,curve,matrix,tips,2cell]{xy}
  \SelectTips{eu}{10} \UseTips
  \UseAllTwocells

\DeclareMathAlphabet\EuR{U}{eur}{m}{n}
\SetMathAlphabet\EuR{bold}{U}{eur}{b}{n}

\makeindex             

\theoremstyle{plain}
\newtheorem{theorem}{Theorem}[section]
\newtheorem{lemma}[theorem]{Lemma}
\newtheorem{proposition}[theorem]{Proposition}

\theoremstyle{definition}

\newtheorem{remark}[theorem]{Remark}

{\catcode`@=11\global\let\c@equation=\c@theorem}

\newcommand{\comsquare}[8]                   
{\begin{CD}
#1 @>#2>> #3\\
@V{#4}VV @V{#5}VV\\
#6 @>#7>> #8
\end{CD}
}

\newcommand{\xycomsquare}[8]                   
{\xymatrix
{#1 \ar[r]^{#2} \ar[d]^{#4} &
#3 \ar[d]^{#5}  \\
#6\ar[r]^{#7} &
#8
}
}

\newcommand{\xycomsquareminus}[8]                      
{\xymatrix{#1 \ar[r]^-{#2} \ar[d]^-{#4} &
#3 \ar[d]^-{#5}  \\
#6\ar[r]^-{#7} &
#8
}
}

\newcommand{\caln}{{\mathcal N}}

\newcommand{\IC}{{\mathbb C}}

\newcommand{\IR}{{\mathbb R}}

\newcommand{\IZ}{{\mathbb Z}}

\newcommand{\im}{\operatorname{im}}

\newcommand{\lead}{\operatorname{lead}}

\newcommand{\pr}{\operatorname{pr}}

\newcommand{\width}{\operatorname{wd}}

\newcommand{\higherlim}[3]{{\setbox1=\hbox{\rm lim}
        \setbox2=\hbox to \wd1{\leftarrowfill} \ht2=0pt \dp2=-1pt
        \mathop{\vtop{\baselineskip=5pt\box1\box2}}
        _{#1}}^{#2}#3}

\newcommand{\version}[1]                       
{\begin{center} last edited on #1\\
last compiled on \today\\
name of texfile: \jobname
\end{center}
}

\newcounter{commentcounter}

\begin{document}

\typeout{----------------------------  spectral.tex  ----------------------------}


\typeout{------------------------------------ Abstract ----------------------------------------}

\begin{abstract}
  We give a polynomial bound on the spectral density function of a matrix over the complex
  group ring of $\IZ^d$.  It yields an explicit lower bound on the Novikov-Shubin
  invariant associated to this matrix showing in particular that the Novikov-Shubin
  invariant is larger than zero.
\end{abstract}

\maketitle


 \typeout{-------------------------------   Section 0: Introduction --------------------------------}

\section{Introduction}


\subsection{Summary}
\label{subsec:summary}

The main result of this paper is that for a $(m,n)$-matrix $A$ over the complex group ring of
$\IZ^d$ the Novikov-Shubin invariant of the bounded $\IZ^d$-equivariant operator 
$r_A^{(2)} \colon L^2(\IZ^d)^m \to L^2(\IZ^d)^n$ given by right multiplication with $A$ is larger
than zero. Actually rather explicit lower bounds in terms of elementary invariants of the
minors of the matrix $A$ will be given.
This is a direct consequence of a polynomial bound of the
spectral density function of $r_A^{(2)} $ which is interesting in its own right.
It will play a role in the forthcoming paper~\cite{Friedl-Lueck(2015twisting)}, where 
we will twist $L^2$-torsion with finite dimensional  representations and it
will be crucial that we allow complex coefficients and not only integral coefficients.

Novikov-Shubin invariants were  originally defined analytically in
\cite{Novikov-Shubin(1986b),Novikov-Shubin(1986a)}. More information about them can be found 
for instance in~\cite[Chapter~2]{Lueck(2002)}.

Before we state the main result, we need the following notions.


\subsection{The width and the leading coefficient}
\label{subsec:The_Width_and_the_leading_coefficient}

Consider a non-zero  element 
$p = p(z_1^{ \pm 1}, \ldots, z_{d}^{\pm 1}) $ in
$\IC[\IZ^d] = \IC[z_1^{ \pm 1}, \ldots, z_{d}^{\pm 1}]$ for some integer $d \ge 1$.

There are integers $n_d^-$ and $n_d^+$  and elements $q_n(z_1^{\pm 1},\ldots ,z_{d-1}^{ \pm 1})$ in
$\IC[\IZ^{d-1}] = \IC[z_1^{ \pm 1}, \ldots, z_{d-1}^{\pm 1}]$ uniquely determined by the properties
that 
\begin{eqnarray*}
n_d^- 
& \le & 
n_d^+;
\\
q_{n_d^-}(z_1^{\pm 1},\ldots ,z_{d-1}^{ \pm 1}) 
& \not= &
0;
\\
q_{n_d^+}(z_1^{\pm 1},\ldots ,z_{d-1}^{ \pm 1})
& \not= &
0;
\\
p(z_1^{ \pm 1}, \ldots, z_{d}^{\pm 1}) 
& = &
\sum_{n = n_d^-}^{n_d^+} 
q_n(z_1^{\pm 1},\ldots ,z_{d-1}^{ \pm 1}) \cdot z_d^n.
\end{eqnarray*}
In the sequel denote
\begin{eqnarray*}
w(p) & = & n_d^+ -n_d^-;
\\
q^+(p) & = & q_{n_d^+}(z_1^{\pm 1},\ldots ,z_{d-1}^{ \pm 1}).
\end{eqnarray*}

Define inductively elements $p_i(z_1^{ \pm 1}, \ldots, z_{d-i}^{\pm 1}) $ in
$\IC[\IZ^{d-i}] = \IC[z_1^{ \pm 1}, \ldots, z_{d-i}^{\pm 1}]$ and integers 
$w_i(p) \ge 0$ for $i = 0,1,2, \ldots, d$ by
\begin{eqnarray*}
p_0(z_1^{ \pm 1}, \ldots, z_{d}^{\pm 1})  
& := & 
p(z_1^{ \pm 1}, \ldots, z_{d}^{\pm 1});
\\
p_1(z_1^{\pm 1},\ldots ,z_{d-1}^{ \pm 1}) 
& := & 
q^+(p)
\\
p_{i} 
& := & 
q^+(p_{i-1}) \quad \text{for } i = 1,2 \ldots, d;
\\
w_0(p) 
& := & w(p)
\\
w_i(p) 
& := & 
w(p_i) \quad \text{for } i = 1,2 \ldots, (d-1).
\end{eqnarray*}

Define the \emph{width} of $p = p(z_1^{ \pm 1}, \ldots, z_{d}^{\pm 1})$ to be
\begin{eqnarray}
\width(p) = \max\{w_0(p), w_1(p), \ldots, w_{d-1}(p)\},
\label{width(p)}
\end{eqnarray}
and the leading \emph{coefficient of $p$} to be
\begin{eqnarray}
\lead(p) & = & p_d.
\label{lead(p)}
\end{eqnarray}

Obviously we have
\begin{eqnarray*}
& \width(p) \ge \width(p_1) \ge \width(p_2) \ge \cdots \ge \width(p_d) = 0; &
\\
& \lead(p) = \lead (p_1) = \ldots = \lead(p_0) \not= 0.
\end{eqnarray*}

Notice that $p_i$, $\width(p)$ and $\lead(p)$ do depend on the ordering of the variables $z_1, \ldots , z_d$.

\begin{remark}[Leading coefficient]\label{rem:leading_coefficients}
The  name ``leading coefficient'' comes from the following alternative definition.
Equip $\IZ^d$ with the lexicographical order, i.e., we put $(m_1, \ldots, m_d) < (n_1,
\ldots, n_d)$, if $m_d < n_d$, or if $m_d = n_d$ and $m_{d-1} < n_{d-1}$, or if $m_d = n_d$, $m_{d-1}
= n_{d-1}$ and $m_{d-2} < n_{d-2}$, or if $\ldots$, or if $m_i = n_i$ for $i = d, (d-1), \ldots, 2$ and
$m_1 < n_1$. We can write $p$ as a finite sum with complex coefficients $a_{n_1, \ldots, n_d}$
\[p(z_1^{\pm}, \ldots, z_d^{\pm}) 
= \sum_{(n_1, \ldots,  n_d) \in \IZ^d} a_{n_1, \ldots, n_d} \cdot z_1^{n_1} \cdot z_2^{n_2} \cdot \cdots \cdot z_d^{n_d}.
\]
Let $(m_1, \ldots m_d) \in \IZ^d$ be maximal with respect to the lexicographical order among those elements
$(n_1, \ldots, n_d) \in \IZ^d$ for which $a_{n_1, \ldots, n_d} \not= 0$. Then  the leading coefficient
of $p$ is $a_{m_1, \ldots, m_d}$.
\end{remark}


\subsection{The $L^1$-norm of a matrix}
\label{subsec:The_L1-norm_of_a_matrix}

For an element $p = \sum_{g \in \IZ^d} \lambda_g \cdot g \in \IC[\IZ^d]$ define
$||p||_1 := \sum_{g \in G} |\lambda_g|$. For a matrix $A \in M_{m,n}(\IC[\IZ^d])$ define
\begin{eqnarray}
||A||_1 & = & 
\max\{||a_{i,j}||_1 \mid 1 \le i \le m, 1 \le j \le n\}.
\label{L1-norm_of_a_matrix}
\end{eqnarray}

The main purpose of this notion is that it gives an a priori  upper bound on the norm $r_A^{(2)} \colon L^2(\IZ^d) \to L^2(\IZ^d)$,
namely, we get from~\cite[Lemma~13.33 on page~466]{Lueck(2002)}
\begin{eqnarray}
||r_A^{(2)}|| & \le & m \cdot n \cdot ||A||_1.
\label{norm_estimate_by_L1-norm}
\end{eqnarray}


\subsection{The spectral density function}
\label{subsec:The_spectral_density_function}

Given $A \in M_{m,n}(\IC[\IZ^d])$, multiplication with $A$ induces a bounded
$\IZ^d$-equivariant operator $r_A^{(2)} \colon L^2(\IZ^d)^m \to L^2(\IZ^d)^n$. We will denote by
\begin{eqnarray}
F\bigl(r_A^{(2)}\bigr) \colon [0,\infty) & \to & [0,\infty)
\label{spectral_density_function}
\end{eqnarray}
its \emph{spectral density function} in the sense
of~\cite[Definition~2.1 on page~73]{Lueck(2002)}, namely, the von Neumann dimension of
the image of the operator obtained by applying the functional calculus to the characteristic function
of $[0,\lambda^2]$ to the operator $(r^{(2)}_A)^*r_A^{(2)}$. In the special  case $m = n = 1$,
where $A$ is given by an element $p \in \IC[\IZ^d]$,
it can be computed in terms of the Haar measure $\mu_{T^d}$ of the $d$-torus $T^d$ 
see~\cite[Example~2.6 on page~75]{Lueck(2002)}
\begin{eqnarray}
F\bigl(r_A^{(2)}\bigr)(\lambda) & = & \mu_{T^d}\bigl(\{(z_1, \ldots, z_d) \in T^d \mid \; |p(z_1, \ldots, z_d)| \le \lambda\}\bigr).
\label{F(P)_in_terms_of_volume}
\end{eqnarray}


\subsection{The main result}
\label{subsec:The_main_result}

Our main result is:

\begin{theorem}[Main Theorem]
\label{the:Main_Theorem}
Consider any natural numbers $d,m,n$ and a non-zero matrix $A \in M_{m,n}(\IC[\IZ^d])$. 
Let $B$ be a quadratic submatrix of $A$ of maximal size $k$ such that the corresponding minor
$p = {\det}_{\IC[\IZ^d]}(B)$ is non-trivial.   Then: 

\begin{enumerate}

\item \label{the:Main_Theorem:spectral_density_estimate}
If $\width(p) \ge 1$, the spectral density function 
of $r_A^{(2)} \colon L^2(\IZ^d)^m \to L^2(\IZ^d)^n$ satisfies for all $\lambda \ge 0$
\begin{multline*}
\lefteqn{F\bigl(r_A^{(2)}\bigr)(\lambda) - F\bigl(r_A^{(2)}\bigr)(0)}
\\ \le 
 \frac{8 \cdot \sqrt{3}}{\sqrt{47}} \cdot k \cdot  d \cdot \width(p)  
\cdot \left(\frac{k^{2k -2} \cdot (||B||_1)^{k-1} \cdot \lambda}{|\lead(p)|}\right)^{\frac{1}{d \cdot \width(p)}}.
\end{multline*}
If $\width(p) = 0$, then $F\bigl(r_A^{(2)}\bigr)(\lambda) = 0$ for all $\lambda < |\lead(p)|$
and $F\bigl(r_A^{(2)}\bigr)(\lambda) = 1$ for all $\lambda \ge  |\lead(p)|$;

\item \label{the:Main_Theorem:Novikov-Shubin}
The Novikov-Shubin invariant of $r_A^{(2)}$ is $\infty$ or $\infty^+$ or a real number satisfying
\[
\alpha\bigl(r_A^{(2)}\bigr) \ge \frac{1}{d \cdot \width(p)},
\]
and is in particular larger than zero.
\end{enumerate}
\end{theorem}

It is known that the Novikov-Shubin invariants of $r_A^{(2)}$ for a matrix $A$  over the integral group ring of $\IZ^d$ 
is  a rational numbers larger than zero unless its value is $\infty$ or $\infty^+$.
This follows from Lott~\cite[Proposition~39]{Lott(1992a)}.
(The author of~\cite{Lott(1992a)}
informed us that his proof of this statement is correct
when $d = 1$ but has a gap when $d > 1$.
The nature of the gap is
described in~\cite[page 16]{Lott(1999b)}.
The proof in this case can be completed
by the same basic method used in~\cite{Lott(1992a)}.) This confirms a 
conjecture of Lott-L\"uck~\cite[Conjecture~7.2]{Lott-Lueck(1995)} for $G = \IZ^d$.
The case of a finitely generated free group $G$ is taken care of by Sauer~\cite{Sauer(2003)}. 

Virtually  finitely generated free abelian groups and virtually finitely generated free groups 
are the only cases of finitely generated groups, where the positivity of the Novikov-Shubin invariants for
all matrices over the complex group ring is now known.
In this context we mention the preprints~\cite{Grabowski(2014),Grabowski-Virag(2013)},
where examples of groups $G$ and matrices $A \in M_{m,n}(\IZ G)$ are constructed for which
the Novikov-Shubin invariant of $r_A^{(2)}$ is zero, disproving a
conjecture of Lott-L\"uck~\cite[Conjecture~7.2]{Lott-Lueck(1995)}.


\subsection{Example}
\label{sec:Example}

Consider the case $d = 2$, $m = 3$ and $n = 2$ and the $(3,2)$-matrix over $\IC[\IZ^2]$
\[
A = 
\begin{pmatrix}
z_1^3 & -1  & 1
\\
2 \cdot z_1 \cdot z_2^2 -16 & z_2 & z_1z_2
\end{pmatrix}
\]
Let $B$ be the $(2,2)$-submatrix obtained by deleting third column.
Then $k = 2$, 
\[
B = 
\begin{pmatrix}
z_1^3 & -1  
\\
2 \cdot z_1 \cdot z_2^2 -16 & z_2
\end{pmatrix}
\] 
and we get
\[
p := {\det}_{\IC[\IZ^2]}(B) = z_1^3 \cdot z_2  + 2 \cdot z_1 \cdot z_2^2 -16.
\]
Using the notation of Section~\ref{subsec:The_Width_and_the_leading_coefficient} 
one easily checks $p_1(z_1)  =  2 \cdot z_1$, $\width(p) = 2$, and $\lead(p)  =  2$.
Obviously $||A||_1 = \max\{|1|,|-1|,|2| + |16|,|1|\} = 18$. Hence
Theorem~\ref{the:Main_Theorem} implies for all $\lambda \ge 0$
\begin{eqnarray*}
F\bigl(r_A^{(2)}\bigr)(\lambda) - F\bigl(r_A^{(2)}\bigr)(0)
& \le  &
 \frac{192\cdot \sqrt{2}}{\sqrt{47}}
\cdot \lambda^{\frac{1}{4}}.
\\
\alpha\bigl(r_A^{(2)}\bigr) 
& \ge &
\frac{1}{4}.
\end{eqnarray*}


\subsection{Acknowledgments}
\label{sec:Acknowledgments.}  This paper is financially supported
by the Leibniz-Preis of the author granted by the Deutsche Forschungsgemeinschaft {DFG}.
The author wants to thank the referee for his useful comments.


\typeout{------------------------  Section 2: The case $m = n = 1$  -------------------------------}

\section{The case $m = n = 1$}
\label{subsec:The_case_m_is_n_is_1}

The main result of this section is the following

\begin{proposition}
\label{pro:case_m_is_n_is_1}
Consider an  non-zero element 
$p$ in $\IC[\IZ^d] = \IC[z_1^{ \pm 1}, \ldots, z_{d}^{\pm 1}]$.
If $\width(p) = 0$, then $F\bigl(r_A^{(2)}\bigr)(\lambda) = 0$ for all $\lambda < |\lead(p)|$
and $F\bigl(r_A^{(2)}\bigr)(\lambda) = 1$ for all $\lambda \ge  |\lead(p)|$. If $\width(p) \ge 1$,
we get for the  spectral density function of $r_p^{(2)}$ for all $\lambda \ge 0$
\[
F\bigl(r_p^{(2)}\bigr)(\lambda) 
\le 
 \frac{8 \cdot \sqrt{3}}{\sqrt{47}} \cdot d \cdot \width(p)  
\cdot \left(\frac{\lambda}{|\lead(p)|}\right)^{\frac{1}{d \cdot \width(p)}}.
\]
\end{proposition}

For the case $d = 1$ and $p$ a monic polynomial, a similar estimate of the shape
$F\bigl(r_p^{(2)}\bigr)(\lambda) \le C_k \cdot \lambda^{\frac{1}{k-1}}$ can be found
in~\cite[Theorem~1]{Lawton(1983)}, where the $k \ge 2$ is the number of non-zero coefficients,
and the sequence of real numbers $(C_k)_{k \ge 2}$ is recursively defined  and satisfies $C_k \ge k-1$.


\subsection{Degree one}
\label{subsec:Proof_of_Proposition_ref(pro:case_m_is_n_is_1)_in_degree_one}

In this subsection we deal with Proposition~\ref{pro:case_m_is_n_is_1} in the case $d = 1$.

We get from the Taylor expansion of $\cos(x)$ around $0$ with the Lagrangian remainder term 
that for any $x \in \IR$ there exists 
$\theta(x) \in [0,1]$ such that
\[
\cos(x) = 1 - \frac{x^2}{2} + \frac{\cos(\theta(x) \cdot x)}{4!} \cdot x^4.
\]
This implies for $x \not= 0$ and $|x| \le 1/2$
\[
\left|\frac{2 - 2 \cos(x)}{x^2} - 1\right|
=
\left|\frac{2 \cdot \cos(\theta(x) \cdot x)}{4!} \cdot x^2\right|
\le 
\left|\frac{2 \cdot \cos(\theta(x) \cdot x)}{4!} \right| \cdot |x|^2
\le \frac{1}{12} \cdot \frac{1}{4} = \frac{1}{48}.
\]
Hence we get for $x \in [-1/2,1/2]$
\begin{eqnarray}
\frac{47}{48} \cdot x^2 \le 2 - 2\cos(x).
\label{cos(x)_versus_x2}
\end{eqnarray}

\begin{lemma}\label{lem:polynomial_F(lambda)-estimate_z-a}
For any complex number $a \in \IZ$ we get for the spectral density function of $(z-a) \in \IC[\IZ] = \IC[z,z^{-1}]$
\[
F\bigl(r_{z-a}^{(2)}\bigr)(\lambda) \le \frac{8 \cdot \sqrt{3}}{\sqrt{47}} \cdot \lambda \quad \text{for}\; \lambda \in [0,\infty).
\]
\end{lemma}
\begin{proof}
We compute using~\eqref{F(P)_in_terms_of_volume}, where $r := |a|$,
\begin{eqnarray*}
F\bigl(r_{z-a}^{(2)}\bigr)(\lambda) 
& = &
\mu_{S^1}\{z \in S^1 \mid |z-a| \le \lambda\}
\\
& = &
\mu_{S^1}\{z \in S^1 \mid |z-r| \le \lambda\}
\\
& = & 
\mu_{S^1}\{\phi \in [-1/2,1/2] \mid |\cos(\phi) + i \sin(\phi) -r| \le \lambda\}
\\
& = & 
\mu_{S^1}\{\phi \in [-1/2,1/2] \mid |\cos(\phi) + i \sin(\phi) -r|^2 \le \lambda^2\}
\\
& = & 
\mu_{S^1}\{\phi \in [-1/2,1/2] \mid (\cos(\phi) -r)^2 + \sin(\phi)^2 \le \lambda^2\}
\\
& = & 
\mu_{S^1}\{\phi \in [-1/2,1/2] \mid r \cdot (2 - 2\cos(\phi) + (r-1)^2 \le \lambda^2\}.
\end{eqnarray*} 
We estimate using~\eqref{cos(x)_versus_x2} for $\phi \in [-1/2,1/2]$
\[
r \cdot (2 - 2\cos(\phi)) + (r-1)^2  \ge r \cdot (2 - 2\cos(\phi)) \ge \frac{47}{48} \cdot \phi^2.
\]
This implies for $\lambda \ge 0$
\begin{eqnarray*}
F\bigl(r_{z-a}^{(2)}\bigr)(\lambda) 
& = & 
\mu_{S^1}\{\phi \in [-1/2,1/2] \mid r \cdot (2 - 2\cos(\phi) + (r-1)^2 \le \lambda^2\}
\\
& \le &
\mu_{S^1}\{\phi \in [-1/2,1/2] \mid \frac{47}{48} \cdot \phi^2 \le \lambda^2\}
\\
& = &
\mu_{S^1}\left\{\phi \in [-1/2,1/2] \;\left|\;  |\phi| \le \sqrt{\frac{48}{47}} \cdot \lambda \right.\right\}
\\
& \le &
2 \cdot \sqrt{\frac{48}{47}} \cdot \lambda 
\\
& = & 
\frac{8 \cdot \sqrt{3}}{\sqrt{47}} \cdot \lambda.
\end{eqnarray*}
\end{proof}

\begin{lemma}\label{pro:case_m_is_n_is_1_in_degree_one}
Let $p(z) \in  \IC[\IZ] = \IC[z,z^{-1}]$ be a non-zero element.
If $\width(p) = 0$, then $F\bigl(r_p^{(2)}\bigr)(\lambda) = 0$ for all $\lambda < |\lead(p)|$
and $F\bigl(r_p^{(2)}\bigr)(\lambda) = 1$ for all $\lambda \ge  |\lead(p)|$. If $\width(p) \ge 1$,
we get 
\[
F\bigl(r_p^{(2)}\bigr)(\lambda) 
\le 
 \frac{8 \cdot \sqrt{3}}{\sqrt{47}} \cdot \width(p)  \cdot \left(\frac{\lambda}{|\lead(p)|}\right)^{\frac{1}{\width(p)}} 
\quad \text{for}\; \lambda \in [0,\infty).
\]
\end{lemma}
\begin{proof}
If $\width(p) = 0$, then $p$ is of the shape $C \cdot z^n$, and the claim follows 
directly from~\eqref{F(P)_in_terms_of_volume}.
Hence we can assume without loss of generality that $\width(p) \ge 1$. We can write $p(z)$ as a product
\[
p(z) = \lead(p) \cdot z^k \cdot \prod_{i=1}^r (z - a_i)
\]
for an integer $r \ge 0$, non-zero complex numbers $a_1, \ldots,  a_r$ and an integer $k$.

Since for any polynomial $p$ and complex number $c \not= 0$ we have for all $\lambda \in [0,\infty)$
\[
F\bigl(r_{c \cdot p}^{(2)}\bigr)(\lambda) = F\bigl(r_p^{(2)}\bigr)\left(\frac{\lambda}{|c|}\right),
\]
we can assume without loss of generality $\lead(p) = 1$. 
If $r = 0$, then $p(z) = z^k$ for some $k \not = 0$ and the claim follows by a direct inspection.
Hence we can assume without loss of generality $r \ge 1$.
Since the width, the leading coefficient 
and the spectral density functions of $p(z)$ and
$z^{-k} \cdot p(z)$ agree, we can assume without loss of generality $k = 0$, or equivalently,
that $p(z)$ has the form for some $r \ge 1$
\[
p(z) = \prod_{i=1}^r (z - a_i).
\]

We proceed by induction over $r$.
The case $r =  1$ is taken care of by
Lemma~\ref{lem:polynomial_F(lambda)-estimate_z-a}. 
The induction step from $r-1\ge 1$ to $r$ is done as follows. 

Put $q(z) =  \prod_{i=1}^{r-1} (z - a_i)$. Then $p(z) = q(z) \cdot (z-a_r)$.
The following inequality for elements $q_1,q_2 \in \IC[z,z^{-1}]$ and $s \in (0,1)$
is a special case of~\cite[Lemma~2.13~(3) on page~78]{Lueck(2002)}
\begin{eqnarray}
F\bigl(r_{q_1 \cdot q_2}^{(2)}\bigr)(\lambda) 
& \le & 
F\bigl(r_{q_1}^{(2)}\bigr)(\lambda^{1-s}) + F\bigl(r_{q_2}^{(2)}\bigr)(\lambda^{s}).
\label{F(pq)_estimated_by_F(p)_and_F(q)}
\end{eqnarray} 
We conclude from~\eqref{F(pq)_estimated_by_F(p)_and_F(q)} applied to 
$p(z) = q(z) \cdot (z-a_r)$ in the special case  $s = 1/r$
\begin{eqnarray*}
F\bigl(r_p^{(2)}\bigr)(\lambda) 
& \le &
F\bigl(r_q^{(2)}\bigr)(\lambda^{\frac{r-1}{r}}) + F\bigl(r_{z-a_r}^{(2)}\bigr)(\lambda^{1/r}).
\end{eqnarray*}
We conclude from the induction hypothesis for $\lambda \in [0,\infty)$
\begin{eqnarray*}
F\bigl(r_q^{(2)}\bigr)(\lambda) 
& \le & 
\frac{8 \cdot \sqrt{3}}{\sqrt{47}} \cdot (r-1) \cdot \lambda^{\frac{1}{r-1}};
\\
 F\bigl(r_{z-a_r}^{(2)}\bigr)(\lambda) 
& \le & 
\frac{8 \cdot \sqrt{3}}{\sqrt{47}} \cdot \lambda.
\end{eqnarray*}
This implies for $\lambda \in [0,\infty)$
\begin{eqnarray*}
F\bigl(r_p^{(2)}\bigr)(\lambda) 
& \le &
F\bigl(r_q^{(2)}\bigr)(\lambda^{\frac{r-1}{r}}) + F\bigl(r_{z-a_r}^{(2)}\bigr)(\lambda^{1/r})
\\
&\le  & 
\frac{8 \cdot \sqrt{3}}{\sqrt{47}} \cdot (r-1) \cdot \left(\lambda^{\frac{r-1}{r}}\right)^{\frac{1}{r-1}}
+ \frac{8 \cdot \sqrt{3}}{\sqrt{47}}  \cdot \lambda^{\frac{1}{r}}
\\
&\le  & 
\frac{8 \cdot \sqrt{3}}{\sqrt{47}} \cdot (r-1) \cdot \lambda^{\frac{1}{r}}
+ \frac{8 \cdot \sqrt{3}}{\sqrt{47}}  \cdot \lambda^{\frac{1}{r}}
\\
& = & 
\frac{8 \cdot \sqrt{3}}{\sqrt{47}} \cdot r \cdot \lambda^{\frac{1}{r}}.
\end{eqnarray*}
\end{proof}


\subsection{The induction step}
\label{subsec:The_induction_step_in_the_proof_of_Proposition_ref(pro:case_m_is_n_is_1)}

Now we finish the proof of Proposition~\ref{pro:case_m_is_n_is_1}
by induction over $d$. 
If $\width(p) = 0$, then $p$ is of the shape $C \cdot z_1^{n_1} \cdot z_2^{n_2} \cdot \cdots \cdot z_d^{n_d}$,
and the claim follows directly from~\eqref{F(P)_in_terms_of_volume}.
Hence we can assume without loss of generality that $\width(p) \ge 1$.
The induction beginning $d =1$ has been taken care of 
by Lemma~\ref{pro:case_m_is_n_is_1_in_degree_one}, the induction step from
$d-1$ to $d \ge 2$ is done as follows.

Since $F\bigl(r_p^{(2)}\bigr)(\lambda) \le 1$, the claim is obviously true for $\frac{\lambda}{|\lead(p)|} \ge 1$.
Hence we can assume in the sequel $\frac{\lambda}{|\lead(p)|} \le 1$.

We conclude from~\eqref{F(P)_in_terms_of_volume} and Fubini's Theorem applied to $T^d = T^{d-1} \times S^1$,
where $\chi_A$ denotes the characteristic function of a subset $A$ and $p_1(z_1^{\pm}, \ldots ,z_{d-1}^{\pm 1})$
has been defined in Subsection~\ref{subsec:The_Width_and_the_leading_coefficient}
\begin{eqnarray}
 \label{F(p)(lambda_estimate:1} & & 
\\
\lefteqn{F\bigl(r_p^{(2)}\bigr)(\lambda)}
& & 
\nonumber
\\ 
& = & 
\mu_{T^d}\bigl(\{(z_1, \ldots , z_d) \in T^d \mid\;  |p(z_1,\ldots, z_d)| \le \lambda\}\bigr)
\nonumber
\\
& = & 
\int_{T^d} \chi_{\{(z_1, \ldots , z_d) \in T^d \mid \; |p(z_1,\ldots, z_d)| \le \lambda\}} \; d\mu_{T^n}
\nonumber
\\
& = & 
\int_{T^{d-1} } \left(\int_{S^1} \chi_{\{(z_1, \ldots , z_d) \in T^d \mid \; |p(z_1,\ldots, z_d)| \le \lambda\}}  \; d\mu_{S^1} \right) \;d\mu_{T^{d-1}}
\nonumber
\\
& = & 
\int_{T^{d-1} } \chi_{\{(z_1, \ldots, z_{d-1}) \in T^{d-1}\mid\; |p_1(z_1, \ldots, z_{d-1}) \le |\lead(p)|^{1/d} \cdot \lambda^{(d-1)1/d}\}} 
\nonumber
\\
& & \hspace{5mm}
\cdot \left(\int_{S^1} \chi_{\{(z_1, \ldots , z_d) \in T^d \mid \; |p(z_1,\ldots, z_d)| \le \lambda\}}  \; d\mu_{S^1} \right) \;d\mu_{T^{d-1}}
\nonumber
\\
& & \hspace{10mm}
+
\int_{T^{d-1} } \chi_{\{(z_1, \ldots, z_{d-1}) \in T^{d-1} \mid\; |p_1(z_1, \ldots, z_{d-1}) > |\lead(p)|^{1/d} \cdot\lambda^{(d-1))/d}\}} 
\nonumber
\\
& & \hspace{15mm} 
\cdot  \left(\int_{S^1} \chi_{\{(z_1, \ldots , z_d) \in T^d \mid \; |p(z_1,\ldots, z_d)| \le \lambda\}}  \; d\mu_{S^1} \right) \;d\mu_{T^{d-1}}
\nonumber
\\
& \le & 
\int_{T^{d-1} } \chi_{(z_1, \ldots, z_{d-1}) \mid\; |p_1(z_1, \ldots, z_{d-1})| \le |\lead(p)|^{1/d} \cdot \lambda^{(d-1)1/d}\}} 
+
\nonumber 
\\
& & \hspace{5mm} 
\max\left. \biggl\{\int_{S^1} \chi_{\{(z_1, \ldots , z_d) \in T^d \mid \; |p(z_1,\ldots, z_d)| \le \lambda\}}  \; d\mu_{S^1} \right| (z_1, \ldots, z_{d-1} ) \in T^{d-1}
\nonumber 
\\
& & \hspace{30mm} 
\; \text{with} \; |p_1(z_1, \ldots, z_{d-1})| > |\lead(p)|^{1/d} \cdot\lambda^{(d-1)/d}\biggr\}.
\nonumber
\end{eqnarray}

We get from the induction hypothesis applied to $p_1(z_1, \ldots, z_{d-1})$
and~\eqref{F(P)_in_terms_of_volume} 
since $\frac{\lambda}{|\lead(p)|} \le 1$, $\width(p_1) \le \width(p)$ and $\lead(p) = \lead(p_1)$

\begin{eqnarray}
& & 
\label{F(p)(lambda_estimate:2}
\\
\lefteqn{\int_{T^{d-1} } \chi_{(z_1, \ldots, z_{d-1}) \mid\; |p_1(z_1, \ldots, z_{d-1})| \le |\lead(p)|^{1/d}\cdot \lambda^{(d-1)1/d}\}}}
& & 
\nonumber
\\
& = & 
\int_{T^{d-1} } \chi_{(z_1, \ldots, z_{d-1}) \mid\; |p_1(z_1, \ldots, z_{d-1})| \le |\lead(p_1)|^{1/d}\cdot \lambda^{(d-1)1/d}\}}
\nonumber
\\
& = & 
F\bigl(r_{p_1}^{(2)}\bigr)\bigl(|\lead(p_1)|^{1/d}| \cdot \lambda^{(d-1)/d}\bigr)
\nonumber \\
& \le &
\frac{8 \cdot \sqrt{3}}{\sqrt{47}} \cdot (d-1) \cdot \width(p_1)  
\cdot \left(\frac{|\lead(p_1)|^{1/d} \cdot \lambda^{(d-1)/d}}{|\lead(p_1)|}\right)^{\frac{1}{(d-1) \cdot \width(p_1)}} 
\nonumber
\\
& = &
\frac{8 \cdot \sqrt{3}}{\sqrt{47}} \cdot (d-1) \cdot \width(p_1)  
\cdot \left(\frac{\lambda}{|\lead(p_1)|}\right)^{\frac{1}{d \cdot \width(p_1)}} 
\nonumber
\\
& = &
\frac{8 \cdot \sqrt{3}}{\sqrt{47}} \cdot (d-1) \cdot \width(p_1)  
\cdot \left(\frac{\lambda}{|\lead(p)|}\right)^{\frac{1}{d \cdot \width(p_1)}} 
\nonumber
\\
& \le &
\frac{8 \cdot \sqrt{3}}{\sqrt{47}} \cdot (d-1) \cdot \width(p)  
\cdot \left(\frac{\lambda}{|\lead(p)|}\right)^{\frac{1}{d \cdot \width(p_1)}} 
\nonumber
\\
& \le  &
\frac{8 \cdot \sqrt{3}}{\sqrt{47}} \cdot (d-1) \cdot \width(p)  
\cdot \left(\frac{\lambda}{|\lead(p)|}\right)^{\frac{1}{d \cdot \width(p)}}.
\nonumber
\end{eqnarray}

Fix  $(z_1, \ldots, z_{d-1} ) \in T^{d-1}$ with $|p_1(z_1, \ldots, z_{d-1})| > \lead(p)^{1/d} \cdot \lambda^{(d-1)/d}$.
Consider the element  $f(z_d^{\pm 1}) := p(z_1, \ldots z_{d-1},z_d^{\pm}) \in \IC[z_d^{\pm}]$. It has the shape
\[
f(z_d^{\pm})
= \sum_{n = n^-}^{n^+} q_n(z_1, \ldots, z_{d-1}) \cdot z_d^n.
\]
The  leading coefficient of $f(z_d^{\pm 1})$ is $p_1(z_1, \ldots z_{d-1}) = q_{n_+}(z_1, \ldots, z_{d-1}) $. 
Hence we get from  Lemma~\ref{pro:case_m_is_n_is_1_in_degree_one} 
applied to $f(z_d^{\pm 1})$ and~\eqref{F(P)_in_terms_of_volume} 
since $\frac{\lambda}{|\lead(p)|} \le 1$, $\width(f) \le \width(p)$ and 
$|\lead(f)| = |p_1(z_1, \ldots z_{d-1}))| > |\lead(p)|^{1/d} \cdot \lambda^{(d-1)/d}$
\begin{eqnarray}
& & 
\label{F(p)(lambda_estimate:3}
\\
\lefteqn{\int_{S^1} \chi_{\{(z_1, \ldots , z_d) \in T^d \mid \; |p(z_1,\ldots, z_d)| \le \lambda\}}  \; d\mu_{S^1}}
& &
\nonumber
\\ 
& = & 
\int_{S^1} \chi_{\{z_d \in S^1 \mid \; |f(z_d)| \le \lambda\}}  \; d\mu_{S^1}
\nonumber
\\
& = &  
\frac{8 \cdot \sqrt{3}}{\sqrt{47}} \cdot  \width(f)  
\cdot \left(\frac{\lambda}{\lead(f)}\right)^{\frac{1}{\width(f)}}
\nonumber 
\\
& \le  &  
\frac{8 \cdot \sqrt{3}}{\sqrt{47}} \cdot  \width(f)  
\cdot \left(\frac{\lambda}{\lead(p)^{1/d} \cdot \lambda^{(d-1)/d}}\right)^{\frac{1}{\width(f)}}
\nonumber 
\\
& = &  
\frac{8 \cdot \sqrt{3}}{\sqrt{47}} \cdot  \width(f)  
\cdot \left(\frac{\lambda}{\lead(p)}\right)^{\frac{1}{d \cdot \width(f)}} 
\nonumber 
\\
& \le  &  
\frac{8 \cdot \sqrt{3}}{\sqrt{47}} \cdot  \width(p)  
\cdot \left(\frac{\lambda}{\lead(p)}\right)^{\frac{1}{d \cdot \width(p)}}.
\nonumber 
\end{eqnarray}
Combining~\eqref{F(p)(lambda_estimate:1},~\eqref{F(p)(lambda_estimate:2} and~\eqref{F(p)(lambda_estimate:3}
yields for $\lambda$ with $\frac{\lambda}{|\lead(p)|} \le 1$
\begin{eqnarray*}
F\bigl(r_p^{(2)}\bigr)(\lambda) 
& \le & 
\frac{8 \cdot \sqrt{3}}{\sqrt{47}} \cdot (d-1) \cdot \width(p)  
\cdot \left(\frac{\lambda}{|\lead(p)|}\right)^{\frac{1}{d \cdot \width(p)}} 
\\
& & \hspace{40mm}+ 
\frac{8 \cdot \sqrt{3}}{\sqrt{47}} \cdot  \width(p)  
\cdot \left(\frac{\lambda}{|\lead(p)|}\right)^{\frac{1}{d \cdot \width(p)}} 
\\ 
& = & 
\frac{8 \cdot \sqrt{3}}{\sqrt{47}} \cdot d \cdot  \width(p)  
\cdot \left(\frac{\lambda}{|\lead(p)|}\right)^{\frac{1}{d \cdot \width(p)}}.
\end{eqnarray*}
This finishes the proof of Proposition~\ref{pro:case_m_is_n_is_1}.



\typeout{---------------  Section 2: Proof of the main Theorem  ---------------------}

\section{Proof of the main Theorem~\ref{the:Main_Theorem}}
\label{sec:Proof_of_Main_Theorem}

Now we can complete the proof of our Main Theorem~\ref{the:Main_Theorem}.
We need the following preliminary result

\begin{lemma}\label{lem:preparation}
Consider $B \in M_{k,k}(\IC[\IZ^d])$ such that  $p:= \det_{\IC[\IZ^d]}(B)$ is non-trivial. 
Then we get for all $\lambda \ge 0$
\[
F\bigl(r_B^{(2)}\bigr)(\lambda) \le k \cdot F\bigl(r_p^{(2)}\bigr)\bigl(||r_B^{(2)}||^{k-1} \cdot \lambda\bigr).
\] 

\end{lemma}
\begin{proof}
In the sequel we will identify $L^2(\IZ^d)$ and $L^2(T^d)$ by the Fourier transformation.
We can choose a unitary $\IZ^d$-equivariant operator $U \colon L^2(\IZ^d)^k \to L^2(\IZ^d)^k$ 
and functions $f_1, f_2, \ldots, f_k \colon T^d\to \IR$
such that $0 \le f_1(z) \le f_2(z) \le \ldots \le f_k(z) $ holds for all $z \in T^d$ and 
we have the following equality of bounded $\IZ^d$-equivariant operators
$L^2(\IZ^d)^k = L^2(T^d)^k  \to L^2(\IZ^d)^k= L^2(T^d)^k$, see~\cite[Lemma~2.2]{Lueck-Roerdam(1993)}
\begin{eqnarray}
(r_B^{(2)})^* \circ r_B^{(2)}  = U \circ 
\begin{pmatrix}
r_{f_1}^{(2)} & 0 & 0  & \cdots &  0 & 0
\\ 
0 & r_{f_2}^{(2)} & 0 &  \cdots & 0 & 0
\\ 
0 & 0 & r_{f_3}^{(2)} &   \cdots & 0 & 0
\\
\vdots & \vdots & \vdots  & \ddots & \vdots & \vdots
\\
0 & 0 & 0  & \cdots & r_{f_{k-1}}^{(2)} & 0 
\\
0 & 0 & 0  &  \cdots & 0 & r_{f_k}^{(2)}  
\end{pmatrix} \circ U^*.
\label{diagonalization}
\end{eqnarray}
Since $p \not= 0$ holds by assumption and hence the rank of $B$ over $\IC[\IZ^d]^{(0)}$ is maximal, 
we conclude from~\cite[Lemma~1.34 on page~35]{Lueck(2002)}
that $r_B^{(2)}$ and hence $r_{f_i}^{(2)}$ for each $i = 1,2, \ldots, k$
are weak isomorphisms, i.e., they are  injective  and have dense images. We conclude 
from~\cite[Lemma~2.11~(11) on page~77 and Lemma~2.13 on page~78]{Lueck(2002)}
\[
F(r_B^{(2)})(\lambda) 
= 
F\left((r_B^{(2)})^* \circ r_B^{(2)} \right)(\lambda^2)
=  
\sum_{i = 1}^k F(r_{f_i}^{(2)})(\lambda^2).
\]
For  $i = 1,2, \ldots, k$ we have $f_1(z) \le f_i(z)$ for all $z \in T^d$ and hence 
$F\bigl(r_{f_i}^{(2)}\bigr)(\lambda) \le F\bigl(r_{f_1}^{(2)}\bigr)(\lambda)$ for all $\lambda \ge 0$. This implies
\begin{eqnarray}
F\bigl(r_B^{(2)}\bigr)(\lambda)  \le k \cdot F\bigl(r_{f_1}^{(2)})(\lambda^2).
\label{estimate_for_spectral_density_after_diagonalization}
\end{eqnarray}

Let $B^* \in M_{k,k}(\IC[\IZ^d)$ be the matrix obtain from $B$ by transposition and applying to each entry the involution
$\IC[\IZ^d] \to \IC[\IZ^d|$ sending $\sum_{g \in G} \lambda_g \cdot g$ to $\sum_{g \in G}\overline{\lambda_g} \cdot g^{-1}$.
Then $\bigl(r_B^{(2)}\bigr)^* = r_{B^*}^{(2)}$.
Since $(r_B^{(2)})^* \circ r_B^{(2)} = r_{BB^*}^{(2)}$ and $\det_{\IC[\IZ^d]}(BB^*) = \det_{\IC[\IZ^d]}(B) \cdot \det_{\IC[\IZ^d]}(B^*)
= p \cdot p^*$ holds, we conclude from~\eqref{diagonalization}  the equality of functions $T^d \to [0,\infty]$
\[
pp^* = \prod_{i = 1}^k f_i.
\]
Since $\sup\{|f_i(z)| \mid z \in T^d\}$ agrees with the operatornorm $||r_{f_i}^{(2}||$
and we have $||r_B^{(2)}||^2 = ||(r_B^{(2)})^* r_B^{(2)}|| = \max\bigl\{||r_{f_i}^{(2)}|| \; \bigl| \; i = 1,2, \ldots, k\} = ||r_{f_k}^{(2)}||$,
we obtain the inequality of functions $T^d \to [0,\infty]$
\[
pp^* \le \left(\prod_{i = 2}^k ||r_{f_i}^{(2)}|| \right) \cdot f_1 \le \bigl(||r_B^{(2)}||^2\bigr)^{k-1} \cdot f_1.
\]
Hence we get for  all $\lambda \ge 0$
\begin{eqnarray*}
F\bigl(r_{pp^*}^{(2)}\bigr)\left(\bigl(||r_B^{(2)}||^{k-1}  \lambda\bigr)^2\right)
& = & 
F\bigl(r_{pp^*}^{(2)}\bigr)\left(||r_B^{(2)}||^2\bigr)^{k-1}  \lambda^2\right)
\\
& \ge & 
F\left((||r_B^{(2)}||^2\bigr)^{k-1} \cdot r_{f_1}^{(2)}\right)\left(||r_B^{(2)}||^2\bigr)^{k-1} \cdot \lambda^2\right)
\\
& = & 
F\bigl(r_{f_1}^{(2)})(\lambda^2).
\end{eqnarray*}
This together with~\eqref{estimate_for_spectral_density_after_diagonalization} 
and~\cite[Lemma~2.11~(11) on page~77]{Lueck(2002)} implies
\begin{eqnarray*}
F\bigl(r_B^{(2)}\bigr)(\lambda)  
&\le &
k \cdot F\bigl(r_{f_1}^{(2)})(\lambda^2)
\\
& \le &
k \cdot F\bigl(r_{pp^*}^{(2)}\bigr)\left(\bigl(||r_B^{(2)}||^{k-1}  \lambda\bigr)^2\right)
\\
& \le &
k \cdot F\bigl(r_{p}^{(2)}\bigr)\bigl(||r_B^{(2)}||^{k-1}  \lambda\bigr).
\end{eqnarray*}
\end{proof}

\begin{proof}[Proof of the Main Theorem~\ref{the:Main_Theorem}]
\eqref{the:Main_Theorem:spectral_density_estimate}
In the sequel we denote by $\dim_{\caln(G)}$ the von Neumann dimension, see for 
instance~\cite[Subsection~1.1.3]{Lueck(2002)}.
The rank of the matrices $A$ and $B$ over the quotient field $\IC[\IZ^d]^{(0)}$ is $k$.
The operator $r_{B}^{(2)} \colon L^2(\IZ^d)^k \to L^2(\IZ^d)^k$ is a weak isomorphism, 
and $\dim_{\caln(\IZ^d)}(\overline{\im(r_A^{(2)}))} = k$ 
because of~\cite[Lemma~1.34~(1) on page~35]{Lueck(2002)}.
In particular we have $F\bigl(r_{B}^{(2)}\bigr)(0) = 0$.

Let $i^{(2)} \colon L^2(\IZ^d)^k \to L^2(\IZ^d)^m $ be the inclusion
corresponding to $I \subseteq \{1,2, \ldots, m\}$ and let 
$\pr^{(2)} \colon L^2(\IZ^d)^n \to L^2(\IZ^d)^k$ be
the projection corresponding to $J \subseteq \{1,2, \ldots, n\}$,
where $I$ and $J$ are the subsets specifying the submatrix $B$.
Then $r_{B}^{(2)}  \colon  L^2(\IZ^d)^k \to L^2(\IZ^d)^k$ 
agrees with the composite
\[
r_{B}^{(2)}  \colon L^2(\IZ^d)^k 
\xrightarrow{i^{(2)}} L^2(\IZ^d)^m
\xrightarrow{r_A^{(2)}} L^2(\IZ^d)^n
\xrightarrow{\pr^{(2)}}L^2(\IZ^d)^k.
\]
Let $p^{(2)} \colon  L^2(\IZ^d)^m \to \ker(r_A^{(2)}))^{\perp}$ be the orthogonal projection
onto the orthogonal complement $\ker(r_A^{(2)})^{\perp} \subseteq L^2(G)^m$
of the kernel of $r_A^{(2)}$.
Let $j^{(2)} \colon \overline{\im(r_A^{(2)})} \to L^2(G)^n$ 
be the inclusion of the closure of the image of $r_A^{(2)}$. 
Let $(r_A^{(2)})^{\perp}  \colon  \ker(r_A^{(2)})^{\perp}  \to \overline{\im(r_A^{(2)})}$ be the 
$\IZ^d$-equivariant bounded operator uniquely determined by 
\begin{eqnarray*}
r_A^{(2)} 
& = & 
j^{(2)} \circ (r_A^{(2)})^{\perp} \circ p^{(2)}.
\end{eqnarray*}
The operator $(r_A^{(2)})^{\perp} $
is a weak isomorphism by construction.  We have the decomposition of the weak isomorphism
\begin{eqnarray}
& & r_{B}^{(2)} = \pr^{(2)} \circ \; r_A^{(2)} \circ i^{(2)} 
 = 
\pr^{(2)} \circ j^{(2)} \circ (r_A^{(2)})^{\perp}  \circ p^{(2)} \circ i^{(2)}.
\label{lem:Estimate_in_terms_of_minors_decomposition}
\end{eqnarray}
This implies that the morphism
$p^{(2)} \circ i^{(2)}  \colon L^2(\IZ^d)^k)\to  \ker(r_A^{(2)})^{\perp}$ 
is injective and  the morphism 
$\pr^{(2)} \circ j^{(2)}  \colon \overline{\im(r_A^{(2)})} \to L^2(\IZ^d)^k$
has dense image. Since we already know
$\dim_{\caln(G)}\bigl(\overline{\im(r_A^{(2)})}\bigr)   = k =  \dim_{\caln(G)}\bigl(L^2(\IZ^d)^k\bigr)$,
the operators $p^{(2)} \circ i^{(2)}  \colon L^2(\IZ^d)^k\to  \ker(r_A^{(2)})^{\perp}$ 
and  $\pr^{(2)} \circ j^{(2)}  \colon \overline{\im(r_A^{(2)})} \to L^2(\IZ^d)$
are weak isomorphisms.   Since the operatornorm of $ \pr^{(2)} \circ j^{(2)}$ 
and of $p^{(2)} \circ i^{(2)}$ is less or equal to $1$, 
we conclude from~\cite[Lemma~2.13 on page~78]{Lueck(2002)} 
and~\eqref{lem:Estimate_in_terms_of_minors_decomposition}
\begin{eqnarray*}
\lefteqn{F\bigl(r_A^{(2)}\bigr)(\lambda) - F\bigl(r_A^{(2)}\bigr)(0)}
& & 
\\
& = &
F\bigl((r_A^{(2)})^{\perp}\bigr)(\lambda) 
\\
& \le  & 
F\bigl(\pr^{(2)} \circ j^{(2)} \circ (r_A^{(2)})^{\perp}  \circ p^{(2)} \circ i^{(2)}\bigr)
\bigl(||\pr^{(2)} \circ j^{(2)}|| \cdot  ||p^{(2)} \circ i^{(2)}|| \cdot \lambda\bigr) 
\\
& =  & 
F\bigl(r_{B}^{(2)}\bigr)\bigl(||\pr^{(2)} \circ j^{(2)}|| \cdot  ||p^{(2)} \circ i^{(2)}|| \cdot \lambda\bigr) 
\\
& \le   & 
F\bigl(r_{B}^{(2)}\bigr)(\lambda).
\end{eqnarray*}
Put $p = \det_{\IC[\IZ^d]}(B)$. 
If $\width(p) = 0$, the claim follows directly from
Proposition~\ref{pro:case_m_is_n_is_1}.
It remains to treat the case $\width(p) \ge 1$.
The last inequality together with~\eqref{norm_estimate_by_L1-norm} applied to $B$,
Proposition~\ref{pro:case_m_is_n_is_1}  applied to $p$ and Lemma~\ref{lem:preparation} applied to $B$ 
yields for $\lambda \ge 0$
\begin{eqnarray*}
\lefteqn{F\bigl(r_A^{(2)}\bigr)(\lambda) - F\bigl(r_A^{(2)}\bigr)(0)}
& & 
\\
& \le & 
F\bigl(r_{B}^{(2)}\bigr)(\lambda)
\\
&\le & 
k \cdot F\bigl(r_p^{(2)}\bigr)\bigl(||r_{B}^{(2)}||^{k-1} \cdot \lambda)
\\
&\le & 
k \cdot  F\bigl(r_p^{(2)}\bigr)\bigl((k^2 \cdot ||B||_1)^{k-1} \cdot \lambda)
\\
& \le &
 \frac{8 \cdot \sqrt{3}}{\sqrt{47}} \cdot k \cdot d \cdot \width(p)  
\cdot \left(\frac{k^{2k -2} \cdot (||B||_1)^{k-1} \cdot \lambda}{|\lead(p)|}\right)^{\frac{1}{d \cdot \width(p)}}.
\end{eqnarray*}
This finishes the proof of assertion~\eqref{the:Main_Theorem:spectral_density_estimate}.
Assertion~\eqref{the:Main_Theorem:Novikov-Shubin} is a direct consequence of 
 assertion~\eqref{the:Main_Theorem:spectral_density_estimate} and the definition of the Novikov-Shubin invariant.
This finishes the proof of Theorem~\ref{the:Main_Theorem}.
\end{proof}


\typeout{---------------  Section 3: Example  ---------------------}



\typeout{-------------------------------------- References  ---------------------------------------}


\begin{thebibliography}{10}

\bibitem{Friedl-Lueck(2015twisting)}
S.~Friedl and W.~L\"uck.
\newblock Twisting {$L^2$}-invariants with finite-dimensional representations.
\newblock in preparation, 2015.

\bibitem{Grabowski(2014)}
L.~Grabowski.
\newblock Group ring elements with large spectral density.
\newblock Preprint, arXiv:1409.3212 [math.GR], 2014.

\bibitem{Grabowski-Virag(2013)}
L.~Grabowski and B.~Vir\'ag.
\newblock Random walks on {L}amplighters via random {S}chr\"odinger operators.
\newblock Preprint, 2013.

\bibitem{Lawton(1983)}
W.~M. Lawton.
\newblock A problem of {B}oyd concerning geometric means of polynomials.
\newblock {\em J. Number Theory}, 16(3):356--362, 1983.

\bibitem{Lott(1992a)}
J.~Lott.
\newblock Heat kernels on covering spaces and topological invariants.
\newblock {\em J. Differential Geom.}, 35(2):471--510, 1992.

\bibitem{Lott(1999b)}
J.~Lott.
\newblock Delocalized ${L}\sp 2$-invariants.
\newblock {\em J. Funct. Anal.}, 169(1):1--31, 1999.

\bibitem{Lott-Lueck(1995)}
J.~Lott and W.~L{\"u}ck.
\newblock ${L}\sp 2$-topological invariants of $3$-manifolds.
\newblock {\em Invent. Math.}, 120(1):15--60, 1995.

\bibitem{Lueck(2002)}
W.~L{\"u}ck.
\newblock {\em {$L\sp 2$}-{I}nvariants: {T}heory and {A}pplications to
  {G}eometry and \mbox{{$K$}-{T}heory}}, volume~44 of {\em Ergebnisse der
  Mathematik und ihrer Grenzgebiete. 3.~Folge. A Series of Modern Surveys in
  Mathematics [Results in Mathematics and Related Areas. 3rd Series. A Series
  of Modern Surveys in Mathematics]}.
\newblock Springer-Verlag, Berlin, 2002.

\bibitem{Lueck-Roerdam(1993)}
W.~L{\"u}ck and M.~R{\o}rdam.
\newblock Algebraic ${K}$-theory of von {N}eumann algebras.
\newblock {\em $K$-Theory}, 7(6):517--536, 1993.

\bibitem{Novikov-Shubin(1986b)}
S.~P. Novikov and M.~A. Shubin.
\newblock Morse inequalities and von {N}eumann ${II}\sb 1$-factors.
\newblock {\em Dokl. Akad. Nauk SSSR}, 289(2):289--292, 1986.

\bibitem{Novikov-Shubin(1986a)}
S.~P. Novikov and M.~A. Shubin.
\newblock Morse inequalities and von {N}eumann invariants of non-simply
  connected manifolds.
\newblock {\em Uspekhi. Matem. Nauk}, 41(5):222--223, 1986.
\newblock in Russian.

\bibitem{Sauer(2003)}
R.~Sauer.
\newblock Power series over the group ring of a free group and applications to
  {N}ovikov-{S}hubin invariants.
\newblock In {\em High-dimensional manifold topology}, pages 449--468. World
  Sci. Publ., River Edge, NJ, 2003.

\end{thebibliography}


\end{document}